\documentclass[11pt, reqno]{amsart} 

\numberwithin{equation}{section}

\usepackage{amsfonts}
\usepackage{amsmath}
\usepackage{easy}
\usepackage{easybmat}
\usepackage{amssymb}
\usepackage{graphicx}
\usepackage{subfig}
\usepackage{url}
\usepackage{verbatim}
\usepackage{enumitem}
\usepackage[margin=1in]{geometry}

\theoremstyle{plain}
\newtheorem{theorem}{Theorem}
\newtheorem{proposition}[theorem]{Proposition}
\newtheorem{lemma}[theorem]{Lemma}
\newtheorem{corollary}[theorem]{Corollary}

\newtheorem{question}[theorem]{Question}

\theoremstyle{definition}
\newtheorem{definition}[theorem]{Definition}

\theoremstyle{remark}
\newtheorem{remark}[theorem]{Remark}

\DeclareMathOperator{\rank}{rank}
\DeclareMathOperator{\nullity}{nullity}
\DeclareMathOperator{\im}{im}

\title[Universally Rigid Framework Attachments]{Universally Rigid Framework Attachments}
\thanks{I would like to thank Dylan Thurston for his tremendous help and support, and Andrew Fanoe for fruitful discussions. This work was partially supported by NSF Grant DMS-0739392}
\author{Kiril Ratmanski}
\address{Department of Mathematics, Columbia University, New York, NY 10027}
\email{kr2330@columbia.edu}
\begin{document}


\maketitle


\begin{abstract}
A \emph{framework} is a graph and a map from its vertices to $\mathbb{R}^d$. A framework is called \emph{universally rigid} if there is no other framework with the same graph and edge lengths in $\mathbb{R}^{d'}$ for \emph{any} ${d'}$. A \emph{framework attachment} is a framework constructed by joining two frameworks on a subset of vertices. We consider an attachment of two universally rigid frameworks that are in general position in $\mathbb{R}^d$. We show that the number of vertices in the overlap between the two frameworks must be sufficiently large in order for the attachment to remain universally rigid.
Furthermore, it is shown that universal rigidity of such frameworks is preserved even after removing certain edges. Given positive semidefinite stress matrices for each of the two initial frameworks, we analytically derive the PSD stress matrices for the combined and edge-reduced frameworks. One of the benefits of the results is that they provide a general method for generating new universally rigid frameworks.
\end{abstract}


\begin{section}{Introduction}\label{sec:intro}

A \emph{framework} in $\mathbb{R}^d$ is a graph embedded in Euclidean $d$-dimensional space.
The graph vertices are assigned point coordinates, the edges are represented as line segments connecting corresponding points, and the edge lengths are determined as the distances between the edge-connected points in $\mathbb{R}^d$. One of the important questions in geometry of frameworks is whether there are other frameworks with the same edge lengths. When searching for such frameworks we want to exclude shape-preserving trivial euclidean transformations such as translations, rotations and reflections, applied to already found frameworks. Thus different frameworks will differ in pairwise distances between some pairs of vertices that are not connected by an edge.
In case there are no different frameworks in $\mathbb{R}^d$, the original framework is called \emph{globally rigid}. Furthermore, if there are no different frameworks in any space of higher dimension, the framework is called \emph{universally rigid}.

Connelly first studied universal rigidity in the context of Cauchy polygons \cite{cn}. Later, in \cite{cnch2}, he set and proved a sufficient condition for an arbitrary framework to be universally rigid. The condition is formulated in terms of \emph{stresses} acting on the edges, and it states that a framework in $\mathbb{R}^d$ is universally rigid if there exists a positive semidefinite stress matrix of nullity $d+1$ (see below for definitions). Recently, Gortler and Thurston \cite{gt} proved that this condition is also necessary for \emph{generic} frameworks, in which the vertices' coordinates are algebraically independent over rationals. Alfakih \cite{ak} then established sufficiency of the condition for the wider class of frameworks that are in \emph{general position}, such that no subset of $d+1$ vertices is affinely dependent. In \cite{gt}, Gortler and Thurston have also suggested an algorithm for searching for such PSD stress matrices based on semidefinite programming. (See \cite{pat}, \cite{LLR} for related work, and \cite{SBoyd} for a general overview of semidefinite programming.)


Although universally rigid frameworks are now characterized by the above conditions, not many results are known about operations on frameworks that preserve universal rigidity, nor are there many known examples of systematically generated universally rigid frameworks. In \cite{cn}, Connelly generated larger Cauchy polygons inductively from smaller ones. Another known framework extension method that preserves universal rigidity is $(d+1)$-lateration \cite{trilat}\cite{aklat}, which is also too specific and not efficient in terms of the number of edges added to the original framework.

In this paper we consider a new problem of characterizing the universal rigidity of attachments of two universally rigid frameworks in general position in $\mathbb{R}^d$. By attachment we mean a combined framework in which the two frameworks share some vertices and in which all edges are preserved (see Figure~\ref{fig:attach3D}). We prove that an attachment is universally rigid if and only if the number of shared vertices (at which the two frameworks attach) is $d+1$ or greater (see Theorem \ref{thm:OmegaMain}). We also show that removing those edges connecting shared vertices that are inherited from only one of the attached frameworks preserves universal rigidity of the attachment. Finally, we derive PSD stress matrices of nullity $d+1$ for both original and edge-reduced attachments explicitly, from the PSD stress matrices that are given for the two joined frameworks.

Thus, similar to frameworks, we can characterize the framework attachments by the stress matrices. Moreover, the stress matrices for universally rigid framework attachments have the same properties as the stress matrices for universally rigid frameworks, and computing such matrices does not involve computational search. These results provide a more general method than $(d+1)$-lateration \cite{aklat} for generating new universally rigid frameworks from two or more arbitrary universally rigid frameworks. It can also be used to analyze complex frameworks by decomposing them into smaller attached frameworks for which universal rigidity is known (or can be more easily established). The results may have a potential application to structural engineering and molecular biology.

The rest of the paper is organized as following. Definitions and preliminaries are given in section \ref{sec:bkgnd}. In section \ref{sec:Main} we formulate and prove the main result which is a necessary and sufficient condition for an attachment to be universally rigid. We also prove universal rigidity of the edge-reduced framework attachment. Section \ref{sec:EdgRedStressMat} contains derivation of PSD stress matrices for each of the two attachments. In Section \ref{sec:Apps} we show that edge-reduced framework attachment is a generalization of $(d+1)$-lateration as a method for creating new universally rigid frameworks possessing PSD stress matrices of nullity $d+1$. We also show that an edge-reduced graph attachment of two graphs whose frameworks are universally rigid in any general position always has a a PSD matrix of nullity $d+1$.
\begin{figure}
\includegraphics[scale=0.8]{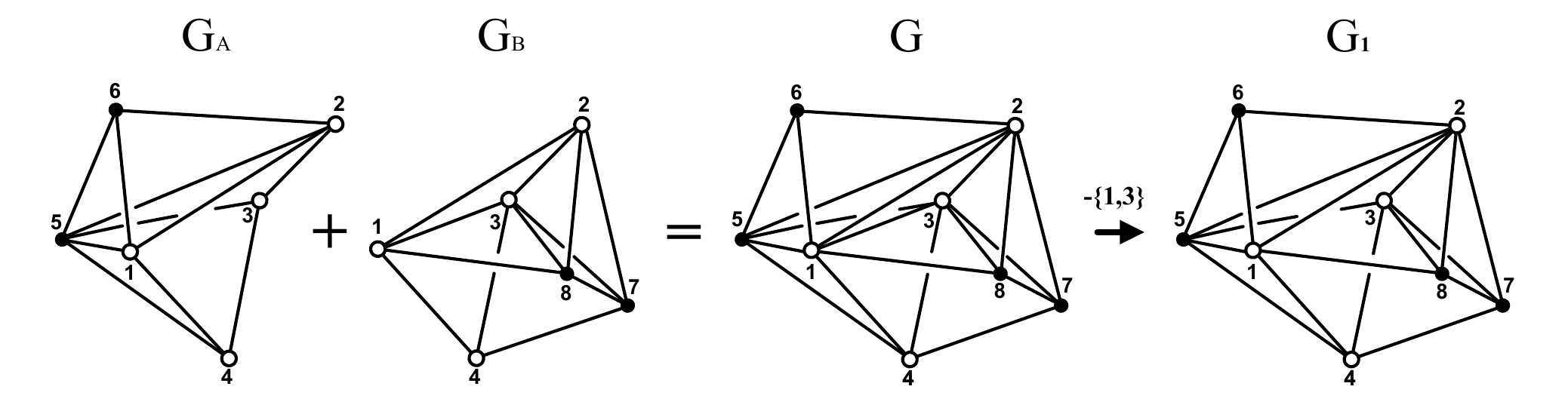}
\caption{Attachment of two universally rigid frameworks in $\mathbb{R}^3$. Vertices $1,2,3$ and $4$ are shared by the two frameworks. Compared to $G_{A}$, graph $G_{B}$ has an additional edge $\{1,3\}$ between the shared vertices. By Theorem \ref{thm:EdgeReducedMain}, removing this edge from the attachment $G$ preserves its universal rigidity.}
\label{fig:attach3D}
\end{figure}

\end{section}


\begin{section}{Background}\label{sec:bkgnd}
A \emph{graph} $G=(V,E)$ is a finite set of vertices $V=(1,2,\ldots,v)$ together with a set of edges $E$. The set $E$ can be represented as a set of two-element subsets $\{i,j\}$ of $V$. We denote the number of graph vertices and edges by $v=|V|$ and $e=|E|$, respectively.
A graph is called \emph{k-connected} if removing at least \emph{k} vertices disconnects the remaining vertices.
Having two graphs $G_{A}=(V_{A},E_{A})$ and $G_{B}=(V_{B},E_{B})$ such that $V_{A}\cap V_{B}\neq\emptyset$ after appropriate labeling of vertices by integers, we can construct a \emph{graph attachment} $G=(V,E)$ where $V=V_{A}\cup V_{B}$ and $E=E_{A}\cup E_{B}$ (Figure~\ref{fig:attach3D}). Throughout the paper, we assume that the set $V_{A}\cap V_{B}$ of shared vertices is a proper subset of both $V_{A}$ and $V_{B}$.

A \emph{configuration} of graph $G=(V,E)$ in $\mathbb{R}^d$ is a mapping of $V$ to $\mathbb{R}^d$. A configuration can also be represented as a single point $\mathbf{p}=(\mathbf{p}_1, \ldots ,\mathbf{p}_v) \in \mathbb{R}^d \times \cdots \times \mathbb{R}^d = \mathbb{R}^{vd}$, where $\mathbf{p}_i$ is the coordinate of vertex $i$ in $\mathbb{R}^d$. A configuration is called \emph{generic} if its coordinates (in $\mathbb{R}^{vd}$) are algebraically independent over $\mathbb{Q}$. A configuration is said to be in \emph{general position} if every subset $\{\mathbf{p}_{i_{1}}, \ldots ,\mathbf{p}_{i_{d+1}}\}$ of $d+1$ vertices is affinely independent. Two configurations $\mathbf{p}$ and $\mathbf{q}$ are \emph{congruent}, or $\mathbf{p} \cong \mathbf{q}$, if $\mathbf{q}=R\mathbf{p}$, where $R$ is an element of the group Euc($d$) of rigid motions in $\mathbb{R}^d$ including translations and rotations. A graph $G$ together with its configuration $\mathbf{p}$ is called \emph{framework} in $\mathbb{R}^d$, denoted as $G(\mathbf{p})$. Two frameworks $G(\mathbf{p})$ and $G(\mathbf{q})$ are said to be \emph{equivalent} if whenever $\{i,j\}$ is an edge of $G$, then $\|\mathbf{p}_{i}-\mathbf{p}_{j}\|=\|\mathbf{q}_{i}-\mathbf{q}_{j}\|$. A \emph{framework attachment} is a framework of graph attachment defined above. Alternatively, if we are given two frameworks both having subsets of vertices with the same pairwise distances, it is possible to coincide the two subsets of vertices in space by applying a rigid motion to one of the frameworks. A framework attachment is created by merging some or all pairs of the coinciding vertices, such that all edges from both frameworks are preserved. There may be more than one attachment constructed from the two frameworks. While the remainder of this section covers single graphs and related notions, we will treat framework attachments in detail in the next sections.

A framework $G(\mathbf{p})$ in $\mathbb{R}^d$ is called \emph{globally rigid} if whenever $G(\mathbf{q})$ is equivalent to $G(\mathbf{p})$ then $\mathbf{p} \cong \mathbf{q}$. A weaker than global rigidity is the notion of local rigidity. A framework $G(\mathbf{p})$ in $\mathbb{R}^d$ is called \emph{locally rigid} if there is $\epsilon>0$ such that whenever $G(\mathbf{q})$ is equivalent to $G(\mathbf{p})$ \emph{and} $\mathbf{q}\in B_{\epsilon}(\mathbf{p})$, then $\mathbf{p} \cong \mathbf{q}$. In other words, locally rigid framework is globally rigid in some $\epsilon$-neighborhood. By definition, global rigidity implies local rigidity.

A framework $G(\mathbf{p})$ in $\mathbb{R}^d$ is called \emph{universally rigid} if whenever $G(\mathbf{q})$ in $\mathbb{R}^{d'}$ is equivalent to $G(\mathbf{p})$ and $d'\geq d$, then $\mathbf{p} \cong \mathbf{q}$. A graph is called \emph{universally rigid} in $\mathbb{R}^d$ if any of its frameworks in general position $\mathbb{R}^d$ is universally rigid. Universal rigidity implies global rigidity.

A necessary condition for global rigidity of a framework is given by the following variation of the Hendrickson's theorem \cite{hn}, which holds for frameworks in general position rather than generic frameworks.
\begin{theorem} \label{thm:globconn}
A globally rigid framework in general position in $\mathbb{R}^{d}$ is $(d+1)$-connected.
\end{theorem}
\begin{proof}
See \cite{hn}, where the same argument for generic frameworks applies to frameworks in general position as well.
\end{proof}

For a given graph $G=(V,E)$ with $v$ vertices, $e$ edges and dimension $d$, we define the \emph{edge function} $f: \mathbb{R}^{vd} \to \mathbb{R}^{e}$ as a function mapping graph configuration in $\mathbb{R}^{vd}$ to its edge-square-length in $G(\mathbf{p})$ (after some ordering of $e$ edges):
$f(\mathbf{p})=f(\mathbf{p}_1, \ldots ,\mathbf{p}_v)=(\ldots ,\frac{1}{2}\|\mathbf{p}_i-\mathbf{p}_j\|^2,\ldots)$. The $e \times vd$ Jacobian matrix corresponding to the edge function is called \emph{rigidity matrix}, denoted by $df$:

\begin{equation}\label{eqn:dfdef}
df=
\left( \begin{array}{ccccccccccc}
\cdot & \cdots & \cdot & \cdots  & \cdot & \cdots & \cdot & \cdots  & \cdot & \cdots & \cdot \\
0     & \cdots & 0     & \mathbf{p}_i-\mathbf{p}_j & 0     & \cdots & 0     & \mathbf{p}_j-\mathbf{p}_i & 0     & \cdots & 0     \\
\cdot & \cdots & \cdot & \cdots  & \cdot & \cdots & \cdot & \cdots  & \cdot & \cdots & \cdot \\
\end{array} \right).
\end{equation}
Given the rigidity matrix $df$, a framework is \emph{infinitesimally rigid} if $\rank (df)=vd-\binom{d+1}{2}$.
Equivalently, infinitesimal rigidity of a framework can be characterized by its infinitesimal flexes. A vector $\mathbf{q}\in\mathbb{R}^{vd}$ is called an \emph{infinitesimal flex} of the framework $G(\mathbf{p})$ if for any edge $\{i,j\}$, the scalar product $(\mathbf{p}_i-\mathbf{p}_j,\mathbf{q}_i-\mathbf{q}_j)=0$. A vector $\mathbf{q}$ is called a \emph{trivial infinitesimal flex} of the framework $G(\mathbf{p})$ if there exists a differentiable path $\mathbf{h}(t)$ of rigid motions in $\mathbb{R}^{d}$ such that $\mathbf{h}(0)=\mathbf{p}$ and $\mathbf{h}'(0)=\mathbf{q}$. Along the path all pairwise distances (and edge lengths in particular) remain constant, therefore evaluating the derivative of $(\mathbf{h}_{i}(t)-\mathbf{h}_{j}(t))^{2}$ at $t=0$ shows that such $\mathbf{q}$ satisfies the scalar product equation above. A framework $G(\mathbf{p})$ is then called infinitesimally rigid if every infinitesimal flex of $G(\mathbf{p})$ is trivial.

The next notion is key for the main results of the paper. For each edge $\{i,j\}$ of graph $G$ we define a scalar $w_{ij}=w_{ji}$, and have them arranged into a vector $w=(\ldots ,w_{ij},\ldots) \in \mathbb{R}^{e}$. We say that vector $w$ is an \emph{equilibrium stress} (or simply \emph{stress} in this paper), if the following vector equation holds for each vertex $i$:

\begin{equation}\label{eqn:stresseq}
\sum_{j:\{i,j\}\in E}w_{ij}(\mathbf{p}_i-\mathbf{p}_j)=0.
\end{equation}
By writing the above as a system of $vd$ scalar equations and then forming into a matrix form, we arrive at the following relation between stresses and the rigidity matrix:

\begin{proposition}\label{prop:stressdf}
The space of stresses equals to $\ker(df^T)$.
\end{proposition}

For each element $w=(\ldots ,w_{ij},\ldots)$ of the space of stresses, there corresponds a $v\times v$ symmetric \emph{stress matrix} $\Omega$ defined such that for $i,j \in V$ and $i\neq j$, $\Omega_{ij}=w_{ij}$ for $\{i,j\}\in E$ and $\Omega_{ij}=0$ for $\{i,j\}\notin E$. The diagonal entries are defined such that the rows and columns sum to zero: $\Omega_{ii}=-\sum_{j\neq i}\Omega_{ij}$. An equivalent definition of the stress matrix, which will be used here, is the following.
\begin{definition}\label{dfn:defStressMat}
A stress matrix of a framework is a $v\times v$ matrix $\Omega$ such that:
\begin{enumerate}
\item for $i,j\in V$, $\Omega_{i,j}=\Omega_{j,i}$;
\item for $i,j\in V$, $i\neq j$ and $\{i,j\}\notin E$, $\Omega_{i,j}=0$;
\item $\sum_{j\in V}\Omega_{i,j}=0$ for all $i\in V$;
\item $\sum_{j\in V}\Omega_{i,j}\cdot\mathbf{p}_j=0$ for all $i\in V$;
\end{enumerate}
\end{definition}
Properties (1)--(3) follow directly from the construction of stress matrix described above. In addition,
$\sum_{j}\Omega_{i,j}\cdot\mathbf{p}_j=\sum_{j\neq i}\Omega_{i,j}\cdot\mathbf{p}_j+
\Omega_{i,i}\cdot\mathbf{p}_i=\sum_{j\neq i}\Omega_{i,j}\cdot\mathbf{p}_j-\sum_{j\neq i}\Omega_{i,j}\cdot\mathbf{p}_i=
\sum_{}w_{ij}(\mathbf{p}_j-\mathbf{p}_i)=0$, which verifies (4).

The following two theorems establish characterization of universal rigidity by the stress matrix.
\begin{theorem} [Connelly \cite{cnch2}, Gortler-Thurston \cite{gt}]\label{thm:UR}
A generic framework $G(\mathbf{p})$ in $\mathbb{R}^{d}$, having at least $d+2$ vertices, is universally rigid if and only if it has a positive semidefinite stress matrix with nullity $d+1$.
\end{theorem}
\begin{theorem} [Alfakih \cite{ak}]\label{thm:AK}
Let $G$ be a framework in general position in $\mathbb{R}^{d}$, having at least $d+2$ vertices. If there exists a positive semidefinite stress matrix of $G$ of nullity $d+1$, then $G$ is universally rigid.
\end{theorem}
We will conclude this section by developing some properties of the stress matrices that will be needed later.
\begin{lemma}\label{lma:p1matrix}
Let $A$ be a $d\times v$ matrix having $\mathbf{p}_1, \ldots ,\mathbf{p}_v$ as its columns. Assume a matrix $B$ is obtained from $A$ by appending a row $(1,\ldots ,1)$. Then $\mathbf{p}$ is in general position if and only if every $(d+1)\times (d+1)$ submatrix of $B$ has full rank.
\end{lemma}
\begin{proof}
If such a submatrix, having columns $\mathbf{p}_{i_{k}}$, $k=1,\ldots ,d+1$, is not full rank, then $\sum{\alpha_{k}}\left(\begin{array}{c}\mathbf{p}_{i_{k}}\\1\end{array}\right)=0$ for some non-trivial $\{\alpha_{k}\}$. But this holds if and only if $\sum{\alpha_{k}\mathbf{p}_{i_{k}}}=0$ and $\sum{\alpha_{k}}=0$, i.e. when $\mathbf{p}_{i_{k}}$ are affinely dependent.
\end{proof}
\begin{corollary}\label{cor:genposopen}
The set of configurations in general position is an open set.
\end{corollary}
\begin{lemma}\label{lma:projker}
If $\Omega$ is a PSD stress matrix of nullity $d+1$, corresponding to a framework $G(\mathbf{p})$ in general position, then the coordinate projections of $\mathbf{p}$ and the vector $(1,\ldots ,1)^T$ are a basis for $\ker(\Omega)$.
\end{lemma}
\begin{proof}
We note that the equation in (4) of Definition \ref{dfn:defStressMat} is a vector equation, and so it is satisfied for each of the coordinate projections. Vector $(1,\ldots ,1)^T$ satisfies the same equation by (3). Taking $d$ coordinate projections' column vectors, vector $(1,\ldots ,1)^T$ and forming into a matrix, we get a matrix $B$ (as in Lemma \ref{lma:p1matrix}) transposed, which is of full rank. Therefore the coordinate projections and vector $(1,\ldots ,1)^T$ are $d+1$ linearly independent vectors in $\ker(\Omega)$, and hence they are the basis.
\end{proof}
In the following lemma we assume a framework in general position has a PSD stress matrix of nullity $d+1$.
\begin{lemma}\label{lma:OmegaDiag}
For a framework in general position, its PSD stress matrix of nullity $d+1$ can be diagonalized by an invertible matrix whose first $d+1$ columns are coordinate projections of $\mathbf{p}$ and a vector $(1,\ldots ,1)^T$.
\end{lemma}
\begin{proof}
Let $\Omega$ be the matrix. Since $\Omega$ is symmetric, $\Omega=U\Lambda U^T$ or, $\Omega U=U\Lambda$, where $U$ is a unitary matrix of eigenvectors corresponding to eigenvalues of $\Omega$ arranged on the diagonal of $\Lambda$ \cite{hj}. Because $\Omega$ is a PSD matrix of nullity $d+1$, $\Lambda$ has $d+1$ zeros and $v-d-1$ positive values on its diagonal. By re-ordering columns of $U$ and corresponding diagonal entries of $\Lambda$ we make the first $d+1$ columns of $U$ to be the eigenvectors corresponding to 0-eigenvalues which are the first $d+1$ entries of the diagonal of $\Lambda$. We then replace these $d+1$ columns of $U$ with (non-orthonormal) set of $d$ coordinate projection vectors and a vector $(1,\ldots ,1)^T$, which are also the eigenvectors corresponding to 0-eigenvalues. This results in a new matrix $S$ such that $\Omega S=S\Lambda$. Since the first $d+1$ columns of $U$ and $S$ are a basis for the same subspace (in particular, each of the first $d+1$ columns of $S$ can be expressed as a linear combination of the first $d+1$ columns of $U$), it is easy to show that the set of all columns of $S$ is linearly independent, which implies that $S$ is invertible. Therefore $\Omega=S\Lambda S^{-1}$.
\end{proof}

For general $v\times v$ positive semidefinite matrices the following is also true:
\begin{lemma}\label{lma:kerSumMat}
Let $A$, $B$ be PSD matrices. Then $\ker(A+B)=\ker(A)\cap\ker(B)$.
\end{lemma}
\begin{proof}
If $x\in\ker(A)\cap\ker(B)$, then $(A+B)x=Ax+Bx=0+0=0$, so $\ker(A)\cap\ker(B)\subseteq\ker(A+B)$.
Conversely, if $x\in\ker(A+B)$, $0=x^{T}(A+B)x=x^{T}Ax+x^{T}Bx$, therefore $x^{T}Ax=x^{T}Bx=0$.
Since $x^{T}Ax=\langle A^{1/2}x,A^{1/2}x\rangle$, where $A^{1/2}=U\Lambda^{1/2} U^{T}$ and hence $\ker(A^{1/2})=\ker(A)$,
we have $0=A^{1/2}x=Ax$, and so $x\in\ker(A)$. Similarly, $x\in\ker(B)$, hence $\ker(A+B)\subseteq\ker(A)\cap\ker(B)$.
\end{proof}

\begin{lemma}\label{lma:omegaepsilon}
For any PSD matrix $\Omega_{1}$ and a matrix $\Omega_{2}$ satisfying $\ker(\Omega_{1})\subseteq\ker(\Omega_{2})$, there is $c > 0$ such that $c\Omega_{1}+\Omega_{2}$ is a PSD matrix with the same nullity as $\Omega_{1}$.
\end{lemma}
\begin{proof}
For a sufficiently small $\epsilon$, we can view $\Omega_{1}+\epsilon\Omega_{2}$ as perturbed matrix $\Omega_{1}$. Assume that $\Omega_{1}$ has rank $r$. Then, the matrix has $r$ positive eigenvalues $\lambda_{l}$ \cite{hj}, while the rest of the eigenvalues are 0. We can assume the eigenvalues are ordered such that $\lambda_{l}\geq\lambda_{l+1}>0$ for $l=1\ldots r-1$. By Weyl theorem \cite{weyl}, the perturbation of the eigenvalues of $\Omega_{1}$ is bounded:
$|\lambda_{l}-\lambda'_{l}|\leq \epsilon\|\Omega_{2} \|$ for all $l$. A sufficient condition for preserving positivity of the first $r$ perturbed eigenvalues is $|\lambda_{l}-\lambda'_{l}|<\lambda_{[r]}$, where $\lambda_{[r]}$ is minimal eigenvalue among the first $r$. Therefore by setting $\epsilon=\frac{\lambda_{[r]}(\Omega_{1})}{2\|\Omega_{2} \|}$, the first $r$ eigenvalues of $\Omega_{1}+\epsilon\Omega_{2}$ are positive.

The rest of the eigenvalues of $\Omega_{1}+\epsilon\Omega_{2}$ are zero. If this was not the case, then  $\rank (\Omega_{1}+\epsilon\Omega_{2})>\rank (\Omega_{1})$, or, $\dim(\ker(\Omega_{1}+\epsilon\Omega_{2}))<\dim(\ker(\Omega_{1}))$. From other side, since $\ker(\Omega_{1})\subseteq\ker(\Omega_{2})=\ker(\epsilon\Omega_{2})$ for $\epsilon>0$, we must have $\ker(\Omega_{1})\subseteq \ker(\Omega_{1}+\epsilon\Omega_{2})$, which implies $\dim(\ker(\Omega_{1}+\epsilon\Omega_{2}))\geq\dim(\ker(\Omega_{1}))$, a contradiction. Therefore $\Omega_{1}+\epsilon\Omega_{2}$ has $r$ positive eigenvalues while the rest are zero, and so the matrix is positive semidefinite with the same nullity as $\Omega_{1}$. By scaling the matrix by $c=\epsilon^{-1}$, the conclusion of the lemma follows.
\end{proof}
\end{section}
\begin{section}{Framework attachments: main results}\label{sec:Main}
In this section we will address the main question of this paper: given two universally rigid frameworks, under which conditions their attachment is also universally rigid?

\begin{theorem}\label{thm:OmegaMain}
A framework attachment of two universally rigid frameworks in general position in $\mathbb{R}^{d}$, not joined on all vertices, is universally rigid if and only if the number of shared vertices is greater or equal $d+1$.
\end{theorem}
\begin{proof}
Assume that we are given the attachment $G$ in $\mathbb{R}^d$ of two universally rigid frameworks, $G_{A}$ and $G_{B}$, that are in general position. Denote the set of vertices in $G_{A}$ by $V_{A}$, the set of vertices of $G_{B}$ by $V_{B}$, and the set of shared vertices by $V_{C}=V_{A}\cap V_{B}$. If $|V_{C}|< d+1$, removing the shared vertices from $G$ leaves a disconnected graph, which means $G$ is not $(d+1)$-connected and, by Theorem \ref{thm:globconn}, not universally rigid.

To prove the converse, we express can express the configuration $\mathbf{p}$ of $G$ as $\mathbf{p}=(\mathbf{p}_{A},\mathbf{p}_{B\backslash C})$, where $\mathbf{p}_{A}$ is the configuration of $G_{A}$. Since $G_{A}$ is universally rigid, any configuration $\mathbf{p'}$ of $G$ satisfies $\mathbf{p}'_{A}\cong \mathbf{p}_{A}$. Therefore for any such $\mathbf{p'}$ there corresponds (by a rigid motion) a configuration $\mathbf{p}''=(\mathbf{p}_{A},\mathbf{p}''_{B\backslash C})$ such that $\mathbf{p''}\cong \mathbf{p'}$. We will now show that any such $\mathbf{p''}$ is congruent to $\mathbf{p}$, thus proving $\mathbf{p'}\cong \mathbf{p}$.

Without loss of generality, assume that the two frameworks are joined at $n\geq d+1$ vertices $v_1,\ldots ,v_n$. Since each framework is in general position, the coordinate vectors $\mathbf{p}_1,\ldots ,\mathbf{p}_{d+1}$ are affinely independent, i.e. $\mathbf{y}_{1}=\mathbf{p}_{2} - \mathbf{p}_{1},\ldots ,\mathbf{y}_{d}=\mathbf{p}_{d+1} - \mathbf{p}_{1}$ are linearly independent.
Since $G_{B}$ is universally rigid, $\mathbf{p}''_{B}\cong \mathbf{p}_{B}$. That means there is a congruency map $H: \mathbf{p}_{B}\to \mathbf{p}''_{B}$. The map fixes $\mathbf{p}_1,\ldots ,\mathbf{p}_{d+1}$, therefore it must be some linear transformation $A$ from the orthogonal group $O_d$ (we can assume, without loss of generality, that vertex $v_1$ is at the origin). The transformation $A$ also fixes the basis vectors $\mathbf{y}_{1},\ldots, \mathbf{y}_{d}$. Therefore, for a coordinate vector $\mathbf{p}_j = \sum_{i}{\alpha_{i} \mathbf{y}_{i}}$ of a vertex $v_j\in V_{B\backslash C}$, we have:
\begin{displaymath}
\mathbf{p}''_j=A \mathbf{p}_j = A\sum_{i}{\alpha_{i} \mathbf{y}_{i}} = \sum_{i}{\alpha_{i} (A\mathbf{y}_{i})} = \sum_{i}{\alpha_{i} \mathbf{y}_{i}} = \mathbf{p}_j,
\end{displaymath}
which shows that $\mathbf{p}''_{B\backslash C}=\mathbf{p}_{B\backslash C}$. Therefore, $\mathbf{p}''=\mathbf{p}$ and hence $\mathbf{p}'\cong\mathbf{p}$. Since $\mathbf{p}'$ is an arbitrary configuration of the attachment, and the above argument did not depend on the dimension $d$, this shows that the attachment is universally rigid.
\end{proof}
With this result, one may naturally ask:
\begin{question}\label{qtn:URGen}
What are the generating graphs for universally rigid graphs under graph attachment?
\end{question}
In other words, we are interested in identifying a smallest set of universally rigid graphs which would span, under finite attachment, the space of universally rigid graph attachments. At a first glance the space may resemble a semigroup with a set of generators. However, since there can be more than one way to attach a pair of graphs, the binary operation is not well defined. Leaving the attempt of finding the answer to the above question for a separate possible research,
we will now establish the second main result, which states that removing certain edges between shared vertices in the attachment preserves universal rigidity (see last diagram in Figure~\ref{fig:attach3D}).
\begin{definition}\label{dfn:EdgeReducedDef}
An \emph{edge-reduced framework attachment} is a framework obtained from framework attachment by removing those edges between the shared vertices that are inherited from only one of the two joined frameworks.
\end{definition}
\begin{theorem}\label{thm:EdgeReducedMain}
An edge-reduced framework attachment of two universally rigid frameworks is universally rigid.
\end{theorem}
\begin{proof}
Given the non edge-reduced attachment $G$ in $\mathbb{R}^d$ and the corresponding edge function $f$, we again denote the set of vertices in $G_{A}$ by $V_{A}$, the set of vertices of $G_{B}$ by $V_{B}$, and the set of shared vertices by $V_{C}$. As in the proof of Theorem \ref{thm:OmegaMain}, here it will also be sufficient to consider only the configurations of $G$ in which the vertices of $G_{A}$ are ``pinned''. Let $\pi$, defined by $\mathbf{p}_{A}=\pi(\mathbf{p})$, be the projection of the attachment configuration $\mathbf{p}=(\mathbf{p}_{A},\mathbf{p}_{B\backslash C})$ onto coordinates of the vertices of $G_{A}$.
Since $G$ is universally rigid (by Theorem \ref{thm:OmegaMain}), the set $f^{-1}(f(\mathbf{p}))$ of all edge-length preserving configurations consists of configurations congruent to $\mathbf{p}$.

Assume now we remove the edge $\{i,j\}$ of length $l_{1}$ between shared vertices that was inherited from $G_{B}$ only, resulting in an edge-reduced framework attachment $G_{1}$ with a corresponding edge function $f_{1}$. Since $G_{A}$ is universally rigid and the removed edge did not belong to $G_{A}$, any configuration $\mathbf{p'}=(\mathbf{p}'_{A},\mathbf{p}'_{B\backslash C})\in f_{1}^{-1}(f_{1}(\mathbf{p}))$ satisfies
$\mathbf{p}'_{A}\cong \mathbf{p}_{A}$. Therefore for any such $\mathbf{p'}$ there corresponds (by a rigid motion) a configuration
$\mathbf{p}''=(\mathbf{p}_{A},\mathbf{p}''_{B\backslash C})\in f_{1}^{-1}(f_{1}(\mathbf{p}))\cap\pi^{-1}(\mathbf{p}_{A})$ such that $\mathbf{p''}\cong \mathbf{p'}$. We will now show that any such $\mathbf{p''}$ is congruent to $\mathbf{p}$, thus proving $\mathbf{p'}\cong \mathbf{p}$.

We make the following observation regarding the configuration sets:
\begin{align*}
f^{-1}(f(\mathbf{p})) &\subseteq f_{1}^{-1}(f_{1}(\mathbf{p})),\\
f^{-1}(f(\mathbf{p})) &= f_{1}^{-1}(f_{1}(\mathbf{p})) \cap D_{ij},
\end{align*}
where $D_{ij}$ is the set of all edge-reduced configurations with the distance between vertices $i$ and $j$ equal to $l_{1}$. By restricting the attachment configurations to those with vertices of $G_{A}$ mapped to $\mathbf{p}_{A}$, we get:
\begin{displaymath}
\mathbf{p}\in f^{-1}(f(\mathbf{p})) \cap \pi^{-1}(\mathbf{p}_{A}) =
f_{1}^{-1}(f_{1}(\mathbf{p})) \cap D_{ij} \cap \pi^{-1}(\mathbf{p}_{A}) =
f_{1}^{-1}(f_{1}(\mathbf{p})) \cap \pi^{-1}(\mathbf{p}_{A})\ni\mathbf{p''},
\end{displaymath}
where the last equality stems from the fact that if $\mathbf{p}_{A}$ is fixed then $dist(i,j)=l_{1}$, or $\pi^{-1}(\mathbf{p}_{A})\subset D_{ij}$, and so the edge length restriction can be omitted. Therefore
$\mathbf{p}\cong \mathbf{p''}\cong \mathbf{p'}$, and since $\mathbf{p'}$ is an arbitrary configuration of $G_{1}$, this shows that $G_{1}$ is universally rigid. Proceeding by induction, we conclude that removing those edges between shared vertices that are inherited only from $G_{B}$ preserves universal rigidity.
\end{proof}
The main idea here is that universal rigidity of $G_{A}$ maintains the same pairwise distances between the shared vertices for all edge-reduced attachment configurations. Since $G$ is also universally rigid, we conclude that all edge-reduced attachments agree on pairwise distances between all vertices, and so, in particular, all configurations are congruent to $\mathbf{p}$.
\end{section}

\begin{section}{Stress matrices for framework attachments}\label{sec:EdgRedStressMat}
Let's assume we are given an attachment of two universally rigid frameworks in general position with PSD stress matrices $\Omega_{A}$ of size $v_{A}\times v_{A}$, and $\Omega_{B}$ of size $v_{B}\times v_{B}$, each of of nullity $d+1$. Here $v_{A}$ and $v_{B}$ are the number of vertices for the corresponding frameworks, and so the framework attachment will have $v=v_{A}+v_{B}-n$ vertices.

\begin{figure}
\includegraphics[scale=0.7]{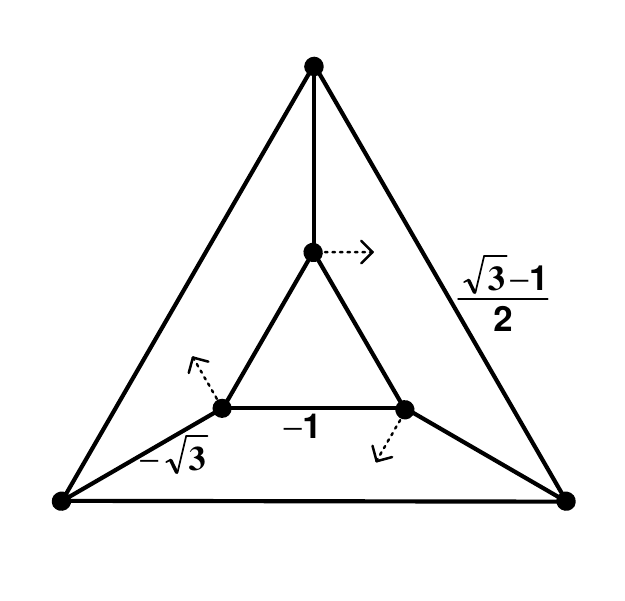}
\caption{A framework in $\mathbb{R}^2$ of a Tutte graph \cite{cnbook}. The framework is universally rigid and has a PSD stress matrix of nullity $3$, but it is not infinitesimally rigid. The arrows indicate a non-trivial infinitesimal flex. The numbers by the edges are the stresses corresponding to a PSD stress matrix of nullity 3.}
\label{fig:Tutte}
\end{figure}

We re-order, if needed, the columns and rows of $\Omega_{A}$ such that the \emph{last} $n$ columns (and rows) correspond to shared vertices. We also rearrange $\Omega_{B}$ to have the \emph{first} $n$ columns (and rows) correspond to shared vertices. Next, we extend framework $G_{A}$ by $v_{B}-n$ non-shared vertices of $G_{B}$ (leaving out new edges), thus getting an extended framework $\widetilde{G}_{A}$. Adding the new disconnected vertices is equivalent to augmenting $\Omega_{A}$ by $v_{B}-n$ rows and columns of zeros stresses, which results in an extended $v\times v$ stress matrix $\widetilde{\Omega}_{A}$. Similarly, we extend framework $G_{B}$ to $\widetilde{G}_{B}$ by adding $v_{A}-n$ non-shared vertices (and no edges) of framework $G_{A}$. The stress matrix $\Omega_{B}$ is extended to a $v\times v$ matrix $\widetilde{\Omega}_{B}$ in a similar fashion except for the way we augment the matrix by zeros, as shown below. (The difference comes from maintaining consistent labeling of vertices.)
\begin{displaymath}
\widetilde{\Omega}_{A}=
\left(
\begin{BMAT}(@,2pt,2pt){c.c}{c.c}
\begin{BMAT}(@,35pt,35pt){c}{c}
\Omega_{A}
\end{BMAT} & 0\\
0 &
\begin{BMAT}(@,25pt,25pt){c}{c}
0
\end{BMAT}
\end{BMAT}
\right)
,\quad
\widetilde{\Omega}_{B}=
\left(
\begin{BMAT}(@,2pt,2pt){c.c}{c.c}
\begin{BMAT}(@,20pt,20pt){c}{c}
0
\end{BMAT}
& 0\\
0 &
\begin{BMAT}(@,40pt,40pt){c}{c}
\Omega_{B}
\end{BMAT}
\end{BMAT}\\
\right).
\end{displaymath}
Clearly, the attachment of the two extended frameworks (joined at all vertices) is identical to the original attachment. Also, the two extended stress matrices are of the same size. We can now state the two main results of this section.
\begin{theorem}\label{thm:OmegaSum}
Given PSD stress matrices $\Omega_{A}$ and $\Omega_{B}$ of nullity $d+1$ for two frameworks in general position sharing $n\geq d+1$ vertices, a PSD stress matrix of nullity $d+1$ for the framework attachment $G$ can be obtained by summing the two matrices after extending each by appropriate number of zero columns and rows:
\begin{equation}\label{eqn:omegaPlus}
\widetilde{\Omega}=\widetilde{\Omega}_{A}+\widetilde{\Omega}_{B}=
\left(
\begin{BMAT}(@,2pt,2pt){ccc}{ccc}
\begin{BMAT}(@,20pt,20pt){c}{c}
\Omega_{A}
\end{BMAT}
& \phantom{0} & 0\\
\phantom{0} &
\begin{BMAT}(@,5pt,5pt){c}{c}
+
\end{BMAT}
& \phantom{0}\\
0 & \phantom{0} &
\begin{BMAT}(@,30pt,30pt){c}{c}
\Omega_{B}
\end{BMAT}
\addpath{(0,1,.)rruu}
\addpath{(1,0,.)uurr}
\end{BMAT}\\
\right).
\end{equation}
\end{theorem}

\begin{theorem}\label{thm:stressReduced}
If $G_{A}$ is also infinitesimally rigid, a positive semidefinite stress matrix of nullity $d+1$ for the edge-reduced framework attachment, in which $K$ edges inherited from $G_{B}$ were removed, is of the form:
\begin{displaymath}
\widetilde\Omega_{re}=c\widetilde\Omega_{A}+\widetilde{\Omega}_{AK}+\widetilde\Omega_{B},
\end{displaymath}
where $c$ is a large enough constant, and the entries of matrix $\widetilde{\Omega}_{AK}$ are derived from the rigidity matrix of $G_{A}$ and stress matrix $\Omega_{B}$ by solving $K$ systems of linear equations.
\end{theorem}

In the rest of this section we prove these two results.
\begin{remark}\label{rmk:noUR}
While being complementary to the two main theorems from the previous section, these results do not actually require universal rigidity of the attached frameworks. From other side, a condition of infinitesimal rigidity of the framework, whose edges are preserved in the edge-reduced attachment, has been introduced for the construction of the stress matrix in the proof of Theorem \ref{thm:stressReduced}. An example in Figure~\ref{fig:Tutte} shows that infinitesimal rigidity is not automatically satisfied even when the framework is universally rigid with a PSD stress matrix of nullity $d+1$. Infinitesimal rigidity does, however, hold for all generic universally rigid frameworks.
\end{remark}

\begin{proof}[Proof of Theorem \ref{thm:OmegaSum}]
The matrix $\widetilde{\Omega}$ is a PSD matrix because both $\widetilde{\Omega}_{A}$ and $\widetilde{\Omega}_{B}$ are. Also, since both $\widetilde{\Omega}_{A}$ and $\widetilde{\Omega}_{B}$ satisfy properties (1), (3) and (4) of Definition \ref{dfn:defStressMat}, so does $\widetilde{\Omega}$. If $i,j\in V_{A}\cup V_{B}$, $i\neq j$ and $\{i,j\}\notin E_{A}\cup E_{B}$  then the ($i,j$) entry of both $\widetilde{\Omega}_{A}$ and $\widetilde{\Omega}_{B}$ is zero, and so is ($i,j$) entry of $\widetilde{\Omega}$. This verifies property (2) as well, and hence the matrix $\widetilde{\Omega}$ is a PSD stress matrix for the framework attachment.

To prove that $\widetilde{\Omega}$ is of nullity $d+1$,
we first consider the case where the number of shared vertices is exactly $d+1$.
Assume, as above, that the last $d+1$ columns(rows) of $\Omega_{A}$, and the first $d+1$ columns(rows) of $\Omega_{B}$ correspond to the shared vertices, in the same order.
We begin by diagonalizing the PSD matrices by $\Omega_{A}=U_{A}\Lambda_{A} U_{A}^T$ and
$\Omega_{B}=S_{B}\Lambda_{B} S_{B}^{-1}$, where $\Omega_{A}$ and $\Omega_{B}$ are diagonal, $U_{A}$ is unitary, and $S_{B}$ is the invertible matrix constructed as in Lemma \ref{lma:OmegaDiag}. The corresponding extended stress matrices, obtained as in the same lemma, are then readily diagonalizable as
$\widetilde{\Omega}_{A}=\widetilde{U}_{A}\widetilde{\Lambda}_{A} \widetilde{U}_{A}^T$ and
$\widetilde{\Omega}_{B}=\widetilde{S}_{B}\widetilde{\Lambda}_{B} \widetilde{S}_{B}^{-1}$, where
the $v\times v$ matrices are of the form:
\begin{displaymath}
\widetilde{U}_{A}=
\left(
\begin{BMAT}(@,2pt,2pt){c.c}{c.c}
\begin{BMAT}(@,35pt,35pt){c}{c}
U_{A}
\end{BMAT} & 0\\
0 &
\begin{BMAT}(@,25pt,25pt){c}{c}
I
\end{BMAT}
\end{BMAT}
\right)
,\quad
\widetilde{S}_{B}=
\left(
\begin{BMAT}(@,2pt,2pt){c.c}{c.c}
\begin{BMAT}(@,20pt,20pt){c}{c}
I
\end{BMAT}
& 0\\
0 &
\begin{BMAT}(@,40pt,40pt){c}{c}
S_{B}
\end{BMAT}
\end{BMAT}\\
\right),
\end{displaymath}
\begin{displaymath}
\widetilde{\Lambda}_{A}=
\left(
\begin{BMAT}(@,2pt,2pt){c.c}{c.c}
\begin{BMAT}(@,35pt,35pt){c}{c}
\Lambda_{A}
\end{BMAT} & 0\\
0 &
\begin{BMAT}(@,25pt,25pt){c}{c}
0
\end{BMAT}
\end{BMAT}
\right)
,\quad
\widetilde{\Lambda}_{B}=
\left(
\begin{BMAT}(@,2pt,2pt){c.c}{c.c}
\begin{BMAT}(@,20pt,20pt){c}{c}
0
\end{BMAT}
& 0\\
0 &
\begin{BMAT}(@,40pt,40pt){c}{c}
\Lambda_{B}
\end{BMAT}
\end{BMAT}\\
\right).
\end{displaymath}

Using the above diagonalization, we can write:
\begin{displaymath}
\widetilde{\Omega}=\widetilde{\Omega}_{A}+\widetilde{\Omega}_{B}=
\widetilde{U}_{A}(\widetilde{\Lambda}_{A}+\widetilde{U}_{A}^{T}\widetilde{\Omega}_{B} \widetilde{U}_{A})\widetilde{U}_{A}^{T}=
\widetilde{U}_{A}(\widetilde{\Lambda}_{A}+\widetilde{U}_{A}^{T}\widetilde{S}_{B}\widetilde{\Lambda}_{B} \widetilde{S}_{B}^{-1}\widetilde{U}_{A})\widetilde{U}_{A}^{T},
\end{displaymath}
or,
\begin{displaymath}
\widetilde{\Omega}=\widetilde{U}_{A}(\widetilde{\Lambda}_{A}+V^{-1}\widetilde{\Lambda}_{B} V)\widetilde{U}_{A}^{T},
\end{displaymath}
where $V=\widetilde{S}_{B}^{-1}\widetilde{U}_{A}$.
Since $\widetilde{U}_{A}$ is invertible, the matrices $(\widetilde{\Lambda}_{A}+V^{-1}\widetilde{\Lambda}_{B} V)$ and $\widetilde{\Omega}$ are of the same nullity. Also, both $\widetilde{\Lambda}_{A}$ and $V^{-1}\widetilde{\Lambda}_{B} V$ are PSD, and so, by Lemma \ref{lma:kerSumMat}, $\nullity(\widetilde{\Omega})=\dim(\ker(\widetilde{\Lambda}_{A}+V^{-1}\widetilde{\Lambda}_{B} V))=
\dim(\ker(\widetilde{\Lambda}_{A})\cap\ker(V^{-1}\widetilde{\Lambda}_{B} V))$. Using the formula for the dimension of the sum of two vector spaces \cite{artin}, we expand the last expression for the intersection, which gives:
\begin{equation}\label{eqn:dimeq}
\nullity(\widetilde{\Omega})=\dim(\ker(\widetilde{\Lambda}_{A}))+\dim(\ker(V^{-1}\widetilde{\Lambda}_{B} V))-\dim(\ker(\widetilde{\Lambda}_{A})+\ker(V^{-1}\widetilde{\Lambda}_{B} V)).
\end{equation}

We will now determine each of the three terms appearing in the above equation. Since only the last $v_{B}$ diagonal entries of $\widetilde{\Lambda}_{A}$ are zero, the basis for $\ker(\widetilde{\Lambda}_{A})$ consist of $v\times 1$ vectors $e_{i}$, where $i=\{v_{A}-d,\ldots ,v\}$ and $e_{i}=(0,\ldots ,0,1,0,\ldots ,0)^{T}$, with 1 at the $i^{th}$ coordinate. In a compact form,

\begin{equation}\label{eqn:kerLa}
\ker(\widetilde{\Lambda}_{A})=\im
\left(
\begin{BMAT}(@,2pt,2pt){c}{c.cc}
\begin{BMAT}(@,10pt,20pt){c}{c}
0
\end{BMAT}
\\
\begin{BMAT}(@,10pt,5pt){c}{c}
\phantom{0}
\end{BMAT}
\\
\begin{BMAT}(@,10pt,30pt){c}{c}
I_{v_{B}}
\end{BMAT}
\end{BMAT}\\
\right)_{v\times v_{B}}.
\end{equation}
Similarly,
\begin{equation}\label{eqn:kerLb}
\ker(\widetilde{\Lambda}_{B})=\im
\left(
\begin{BMAT}(@,2pt,2pt){c}{cc.c}
\begin{BMAT}(@,10pt,20pt){c}{c}
I_{v_{A}}
\end{BMAT}
\\
\begin{BMAT}(@,10pt,5pt){c}{c}
\phantom{0}
\end{BMAT}
\\
\begin{BMAT}(@,10pt,30pt){c}{c}
0
\end{BMAT}
\end{BMAT}\\
\right)_{v\times v_{A}},
\end{equation}
and since $\dim(\ker(V^{-1}\widetilde{\Lambda}_{B} V))=\dim(\ker(\widetilde{\Lambda}_{B}))$, we get:

\begin{align}
\label{eqn:dimkerLa}
\dim(\ker(\widetilde{\Lambda}_{A})) &= v_{B},\\
\label{eqn:dimkerVLbV}
\dim(\ker(V^{-1}\widetilde{\Lambda}_{B} V)) &= v_{A}.
\end{align}

For the third term in (\ref{eqn:dimeq}), we will find the basis of $\ker(\widetilde{\Lambda}_{A})+\ker(V^{-1}\widetilde{\Lambda}_{B} V)$. We first note that $x\in\ker(\widetilde{\Lambda}_{B})\iff V^{-1}x\in\ker(V^{-1}\widetilde{\Lambda}_{B} V)$, and since $V$ is non-singular, the basis $\mathbf{B}$ of $\ker(V^{-1}\widetilde{\Lambda}_{B} V)$ consist of the columns of the matrix-product of $V^{-1}$ and the basis-matrix of $\ker(\widetilde{\Lambda}_{B})$, given in (\ref{eqn:kerLb}):
\begin{displaymath}
\mathbf{B}=V^{-1}
\left(
\begin{BMAT}(@,2pt,2pt){c}{cc.c}
\begin{BMAT}(@,10pt,20pt){c}{c}
I_{v_{A}}
\end{BMAT}
\\
\begin{BMAT}(@,10pt,5pt){c}{c}
\phantom{0}
\end{BMAT}
\\
\begin{BMAT}(@,10pt,30pt){c}{c}
0
\end{BMAT}
\end{BMAT}\\
\right)=\widetilde{U}_{A}^{T}\widetilde{S}_{B}
\left(
\begin{BMAT}(@,2pt,2pt){c}{cc.c}
\begin{BMAT}(@,10pt,20pt){c}{c}
I_{v_{A}}
\end{BMAT}
\\
\begin{BMAT}(@,10pt,5pt){c}{c}
\phantom{0}
\end{BMAT}
\\
\begin{BMAT}(@,10pt,30pt){c}{c}
0
\end{BMAT}
\end{BMAT}\\
\right)=\widetilde{U}_{A}^{T}
\left(
\begin{BMAT}(@,2pt,2pt){c.c}{c.c}
\begin{BMAT}(@,20pt,20pt){c}{c}
I
\end{BMAT}
& 0\\
0 &
\begin{BMAT}(@,40pt,40pt){c}{c}
S_{B}
\end{BMAT}
\end{BMAT}\\
\right)
\left(
\begin{BMAT}(@,2pt,2pt){c}{cc.c}
\begin{BMAT}(@,10pt,20pt){c}{c}
I_{v_{A}}
\end{BMAT}
\\
\begin{BMAT}(@,10pt,5pt){c}{c}
\phantom{0}
\end{BMAT}
\\
\begin{BMAT}(@,10pt,30pt){c}{c}
0
\end{BMAT}
\end{BMAT}\\
\right).
\end{displaymath}
By splitting $S_{B}$ and then $\widetilde{U}_{A}^{T}$ into blocks of appropriate sizes to perform block matrix multiplication, we get:
\begin{equation}\label{eqn:Bv}
\mathbf{B} = \widetilde{U}_{A}^{T}
\left(
\begin{BMAT}(@,2pt,2pt){ccc}{ccc}
\begin{BMAT}(@,20pt,20pt){c}{c}
I
\end{BMAT}
& 0 & 0\\
0 &
\begin{BMAT}(@,5pt,5pt){c}{c}
D
\end{BMAT}
& Y\\
0 & X &
\begin{BMAT}(@,20pt,30pt){c}{c}
Z
\end{BMAT}
\addpath{(0,1,.)rrr}
\addpath{(0,2,.)rrr}
\addpath{(1,0,.)uuu}
\addpath{(2,0,.)uuu}
\end{BMAT}\\
\right)
\left(
\begin{BMAT}(@,2pt,2pt){c}{cc.c}
\begin{BMAT}(@,2pt,20pt){c}{c}
I_{v_{A}}
\end{BMAT}
\\
\begin{BMAT}(@,2pt,5pt){c}{c}
\phantom{0}
\end{BMAT}
\\
\begin{BMAT}(@,2pt,30pt){c}{c}
0
\end{BMAT}
\end{BMAT}\\
\right)
=
\left(
\begin{BMAT}(@,2pt,2pt){c.c}{c.c}
\begin{BMAT}(@,30pt,35pt){c}{c}
U_{A}^{T}
\end{BMAT} & 0\\
0 &
\begin{BMAT}(@,20pt,30pt){c}{c}
I
\end{BMAT}
\end{BMAT}
\right)
\left(
\begin{BMAT}(@,2pt,2pt){cc}{ccc}
\begin{BMAT}(@,20pt,20pt){c}{c}
I
\end{BMAT}
& 0\\
0 &
\begin{BMAT}(@,5pt,5pt){c}{c}
D
\end{BMAT}\\
0 &
\begin{BMAT}(@,5pt,30pt){c}{c}
X
\end{BMAT}
\addpath{(0,1,.)rr}
\addpath{(1,0,.)uuu}
\addpath{(0,2,.)rr}
\end{BMAT}\\
\right)=
\left(
\begin{BMAT}(@,2pt,2pt){c}{c.c}
\begin{BMAT}(@,20pt,30pt){c}{c}
U_{A}^{T}\left(
\begin{BMAT}(@,15pt,15pt){cc}{cc}
I & 0\\
0 & D
\end{BMAT}\right)
\end{BMAT}\\
\begin{BMAT}(@,30pt,30pt){cc}{c}
0 & X
\end{BMAT}
\end{BMAT}\right),
\end{equation}
where the matrix $D$ is the $(d+1)\times (d+1)$ upper left block of the matrix $S_{B}$. By construction, the first $d+1$ columns of $S_{B}$ are the $d$ coordinate
projections of $G_{B}$ and a vector $(1,\ldots ,1)^T$, and so the first $d$ columns of $D$ consist
of coordinates (of shared vertices) while the last column is the $(d+1)\times 1$ vector $(1,\ldots ,1)^T$. Since $G_{B}$ is in general position, by Lemma \ref{lma:p1matrix} the matrix $D$ is invertible.
This in turn implies that the above $v_{A}\times v_{A}$ submatrix $U_{A}^T\left( \begin{array}{cc} I&0\\ 0&D\end{array} \right)$ is invertible. Therefore, right-mulitplication of $\mathbf{B}$, given in (\ref{eqn:Bv}), by the inverse of the submatrix produces a new basis-matrix \cite{artin} for $\ker(V^{-1}\widetilde{\Lambda}_{B} V)$:

\begin{equation}\label{eqn:kerVLbV}
\ker(V^{-1}\widetilde{\Lambda}_{B} V)=\im
\left(
\left(
\begin{BMAT}(@,2pt,2pt){c}{c.c}
\begin{BMAT}(@,30pt,30pt){c}{c}
U_{A}^{T}\left(
\begin{BMAT}(@,15pt,15pt){cc}{cc}
I & 0\\
0 & D
\end{BMAT}\right)
\end{BMAT}\\
\begin{BMAT}(@,30pt,30pt){cc}{c}
0 & X
\end{BMAT}
\end{BMAT}\right)
\left(
\begin{BMAT}(@,15pt,15pt){cc}{cc}
I & 0\\
0 & D
\end{BMAT}\right)^{-1}U_{A}
\right)=\im
\left(
\begin{BMAT}(@,2pt,2pt){c}{cc.c}
\begin{BMAT}(@,10pt,20pt){c}{c}
I_{v_{A}}
\end{BMAT}
\\
\begin{BMAT}(@,10pt,5pt){c}{c}
\phantom{0}
\end{BMAT}
\\
\begin{BMAT}(@,10pt,30pt){c}{c}
Q
\end{BMAT}
\end{BMAT}\\
\right)_{v\times v_{A}}
\end{equation}
for some matrix $Q$. Combining (\ref{eqn:kerLa}) and (\ref{eqn:kerVLbV}) now gives:
\begin{displaymath}
\ker(\widetilde{\Lambda}_{A})+\ker(V^{-1}\widetilde{\Lambda}_{B} V)=\im
\left(
\begin{BMAT}(@,2pt,2pt){c}{c.cc}
\begin{BMAT}(@,10pt,20pt){c}{c}
0
\end{BMAT}
\\
\begin{BMAT}(@,10pt,5pt){c}{c}
\phantom{0}
\end{BMAT}
\\
\begin{BMAT}(@,10pt,30pt){c}{c}
I_{v_{B}}
\end{BMAT}
\end{BMAT}\\
\right)
+\im
\left(
\begin{BMAT}(@,2pt,2pt){c}{cc.c}
\begin{BMAT}(@,10pt,20pt){c}{c}
I_{v_{A}}
\end{BMAT}
\\
\begin{BMAT}(@,10pt,5pt){c}{c}
\phantom{0}
\end{BMAT}
\\
\begin{BMAT}(@,10pt,30pt){c}{c}
Q
\end{BMAT}
\end{BMAT}\\
\right)=\im(I_{v})=\mathbb{R}^{v},
\end{displaymath}
and hence
\begin{equation}\label{eqn:dimsumker}
\dim(\ker(\widetilde{\Lambda}_{A})+\ker(V^{-1}\widetilde{\Lambda}_{B} V))=v=v_{A}+v_{B}-d-1.
\end{equation}
By substituting (\ref{eqn:dimkerLa}), (\ref{eqn:dimkerVLbV}) and (\ref{eqn:dimsumker}) into the dimension formula (\ref{eqn:dimeq}), we finally get:
\begin{displaymath}
\nullity(\widetilde{\Omega})=v_{B}+v_{A}-(v_{A}+v_{B}-d-1)=d+1,
\end{displaymath}
which completes the proof for the case where the number of shared vertices is $d+1$.

In case the number of shared vertices is $n>d+1$, we consider an ``intermediate'' framework attachment which has $d+1$ shared vertices and $n-d-1$ pairs of vertices that are coinciding in space. The attachment with $n$ shared vertices is then constructed by merging each pair of coinciding vertices into one (shared) vertex. Such merging in turn corresponds to certain operations on the stress matrix of the intermediate attachment, which we will now describe in detail.

The PSD stress matrix for the intermediate attachment is similar to (\ref{eqn:omegaPlus}):

\begin{displaymath}
\begin{BMAT}(@,2pt,2pt){cc}{cc}
\phantom{0} \phantom{0} \phantom{0} \phantom{0}
\begin{BMAT}(@,2pt,2pt){ccccc}{c}
\phantom{0} & c_{1} & \phantom{0} & c_{2} & \phantom{0}
\end{BMAT}
& \phantom{0} \\
\widetilde{\Omega}=\left(
\begin{BMAT}(@,2pt,2pt){ccccc}{ccccc}
\Omega_{A} & | & \phantom{0} & | & 0\\
- & + & - & + & -\\
\phantom{0} & | & + & | & \phantom{0}\\
- & + & - & + & -\\
0 & | & \phantom{0} & | & \Omega_{B}
\addpath{(2,0,.)uuurrr}
\addpath{(0,2,.)rrruuu}
\end{BMAT}\right)
&
\begin{BMAT}(@,2pt,2pt){c}{ccccc}
\phantom{0} \\
r_{1} \\
\phantom{0} \\
r_{2} \\
\phantom{0} \\
\end{BMAT}\\
\end{BMAT},
\end{displaymath}
where a pair of columns $c_{1}$, $c_{2}$ (and rows $r_{1}$, $r_{2}$) correspond to one of the $n-d-1$ pairs of coinciding vertices that are to be merged.
The matrix $\widetilde{\Omega}$ is of nullity $d+1$ (which, as in the proof above, comes from $G_{B}$ being in general position). We will now obtain a stress matrix for the attachment with the two coinciding vertices merged, and show that this matrix is positive semidefinite of nullity $d+1$.

We first add column $c_{2}$ to column $c_{1}$, and row $r_{2}$ to row $r_{1}$. Such operation results in a matrix $\widetilde{\Omega}'=R\widetilde{\Omega} R^T$ where $R$ is an elementary matrix, and so $\widetilde{\Omega}'$ is a PSD matrix of nullity $d+1$. Adding the columns and rows corresponds to merging the two vertices. We now exchange the column $c_{2}$ (row $r_{2}$) with the last column (row), resulting in another, positive semidefinite matrix $\widetilde{\Omega}''$ of nullity $d+1$. We can view $\widetilde{\Omega}''$ as a matrix $\widetilde{\Omega}_{m}$ bordered with a vector $y$ and a real number $a$:

\begin{displaymath}
\widetilde{\Omega}''=
\left(
\begin{BMAT}(@,2pt,2pt){c.c}{c.c}
\begin{BMAT}(@,50pt,50pt){c}{c}
\widetilde{\Omega}_{m}
\end{BMAT} & y\\
y^T &
\begin{BMAT}(@,10pt,10pt){c}{c}
a
\end{BMAT}
\end{BMAT}
\right).
\end{displaymath}
By theorem 4.3.8 in \cite{hj}, the eigenvalues $\mu_{i}$ of $\widetilde{\Omega}_{m}$ are related to the eigenvalues $\lambda_{i}$ of $\widetilde{\Omega}''$ by the following inequality:
\begin{displaymath}
\lambda_{1}\leq\mu_{1}\leq\lambda_{2}\leq\mu_{2}\leq\cdots\leq\lambda_{v-1}\leq\mu_{v-1}\leq\lambda_{v}.
\end{displaymath}
Since $\widetilde{\Omega}''$ is of nullity $d+1$, we have $\lambda_{1}=\cdots =\lambda_{d+1}=0$, and hence $\mu_{1}=\cdots =\mu_{d}=0$, $0\leq\mu_{d+1}\leq\lambda_{d+2}$,
and $\mu_{d+2},\ldots ,\mu_{v-1}>0$. This implies that $\widetilde{\Omega}_{m}$ is of nullity $d$ or $d+1$. However, $\widetilde{\Omega}_{m}$ is a stress matrix for a framework attachment (with the pair of coinciding vertices merged), and as such its minimal nullity is $d+1$. This comes from the fact that the configuration corresponding to the attachment of two frameworks in general position with at least $d+1$ vertices each does not lie in any affine subspace of $\mathbb{R}^d$ \cite{gt}.
Repeating the above operations on the $n-d-2$ rows/columns of $\widetilde{\Omega}_{m}$ corresponding to the rest of coinciding vertices, we then obtain a PSD stress matrix of nullity $d+1$ for the attachment with $n$ shared vertices. This completes the proof of Theorem \ref{thm:OmegaSum}.
\end{proof}

Before formally proving Theorem \ref{thm:stressReduced}, we first outline the proof for the case where only \emph{one} edge (say, $\{i,j\}$) of $G_{B}$ is removed from the attachment. We start by having two frameworks and the edge $\{i,j\}$ still connected in $G_{B}$. From the PSD stress matrix $\Omega_{B}$, we obtain a stress $w_{1}$ on the edge $\{i,j\}$. We then \emph{add} the edge $\{i,j\}$ to the same pair of vertices in $G_{A}$, thus obtaining edge-extended framework $\overline{G}_{A}$, and use Proposition \ref{prop:stressw1} below to derive a stress vector $\overline{w}$ such that the stress on the added edge is $-w_{1}$ (as in Figure~\ref{fig:figRedEdge}). With the help of Lemma \ref{lma:omegaepsilon}, we then obtain a PSD stress matrix of nullity $d+1$ bearing the same stress on the added edge. Attaching the two frameworks, by Theorem \ref{thm:OmegaSum}, will correspond to summing the corresponding extended PSD stress matrices, resulting in $w_{1}+(-w_{1})=0$ stress on the edge $\{i,j\}$ of the framework attachment. By Theorem \ref{thm:EdgeReducedMain}, we can remove this edge from the attachment. Since the entry in the stress matrix corresponding to $\{i,j\}$ is 0, after removing this edge the same stress matrix will satisfy equilibrium equation \ref{eqn:stresseq}, and hence it is the stress matrix for the edge-reduced attachment as well. In the proof of Theorem \ref{thm:stressReduced} below, we will follow this outline for a more general case where $K$ edges are removed from the attachment. But first we establish the following preliminary result guaranteeing a non-zero stress on the added edge.
\begin{figure}
\subfloat[]{\includegraphics[scale=1]{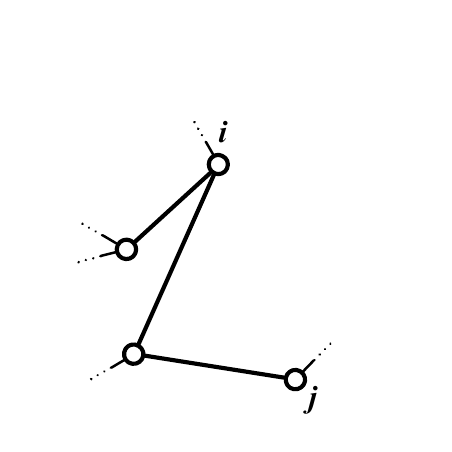}\label{fig:figGac}}\hspace{1in}
\subfloat[]{\includegraphics[scale=1]{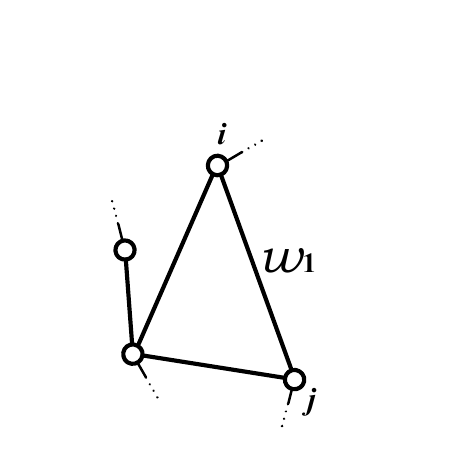}\label{fig:figGbc}}\hspace{1in}
\subfloat[]{\includegraphics[scale=1]{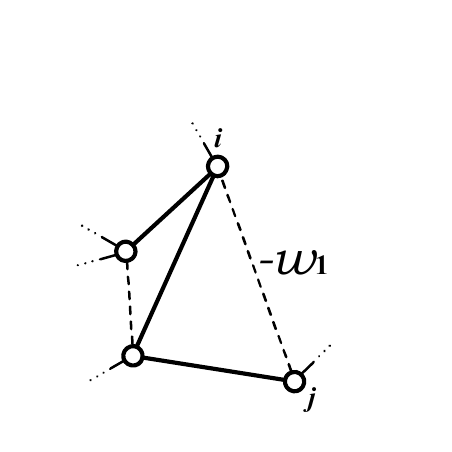}\label{fig:figGabc}}

\caption{Edges between the shared vertices in $G_{A}$ (\ref{fig:figGac}), $G_{B}$ (\ref{fig:figGbc}) and $\overline{G}_{A}$ (\ref{fig:figGabc}). The stresses on the added edges in $\overline{G}_{A}$ counter-balance the stresses on the same edges in $G_{B}$.}

\label{fig:figRedEdge}
\end{figure}

\begin{proposition}\label{prop:stressw1}
For a framework $\overline{G}_{A}$, obtained from infinitesimally rigid framework $G_{A}$ by adding edge $\{i,j\}$, there exists a stress vector $\overline{w}$ such that the stress on the added edge is equal to $-w_{1}$.
\end{proposition}
\begin{proof}
For the unmodified $G_{A}$ we have the rigidity matrix $df$, defined by (\ref{eqn:dfdef}). By proposition \ref{prop:stressdf}, a stress vector is an element of a kernel of rigidity matrix transposed. Since $G_{A}$ is infinitesimally rigid,  $\rank (df)=v_{A}d-\binom{d+1}{2}$. With $e_{A}$ being the number of edges (and the number of rows in $df$), $\nullity(df^{T})=e_{A}-\rank(df)$. Adding edge $\{i,j\}$ to $G_{A}$ corresponds to adding a row to $df$ (or, a column to $df^{T}$), resulting in a new rigidity matrix $\overline{df}$. Since adding an edge preserves infinitesimal rigidity, $\rank (\overline{df})=v_{A}d-\binom{d+1}{2}=\rank (df)$, and so $\nullity(\overline{df}^{T})=e_{A}+1-\rank(df)=\nullity(df^{T})+1$. In other words, adding an edge to an infinitesimally rigid framework increases the dimension of the space of stresses by 1.
We will now show that in the new space of stresses, there exists a vector with non-zero stress component corresponding to the added edge $\{i,j\}$. We first note that columns of $df^{T}$ and components of a stress vector $w\in\ker(df^{T}))$ are indexed by existing edges of $G_{A}$.
We append a new edge related column (in $df^{T}$) and a component (in $w$), and from the following equalities:
\begin{displaymath}
df^{T}w=
\left(
\begin{BMAT}(@,2pt,2pt){c.c}{c}
df^{T} & \rho
\end{BMAT}
\right)
\left(
\begin{BMAT}(@,2pt,2pt){c}{c.c}
w\\
0
\end{BMAT}\\
\right)
=\overline{df}^{T}\cdot
\left(
\begin{BMAT}(@,2pt,2pt){c}{cc}
w\\
0
\end{BMAT}\\
\right),
\end{displaymath}
where $\rho = (0,\ldots \mathbf{p}_{i}-\mathbf{p}_{j},\ldots ,\mathbf{p}_{j}-\mathbf{p}_{i},\ldots ,0)^{T}$ corresponds to the added edge $\{i,j\}$ (see Figure \ref{fig:figGabc}), we have:
\begin{displaymath}
w\in\ker(df^{T}) \iff
\left(
\begin{BMAT}(@,2pt,2pt){c}{cc}
w\\
0
\end{BMAT}\\
\right)
\in\ker(\overline{df}^{T}).
\end{displaymath}
Since $\dim(\ker(\overline{df}^{T}))>\dim(\ker(df^{T}))$, vectors $(w,0)^{T}$ do not span all of $\ker(\overline{df}^{T})$, and so there exists a stress vector $\overline{w}=(w,\psi)^{T}\in \ker(\overline{df}^{T})$ such that $\psi\neq 0$. By scaling, we can set $\psi=-w_{1}$, while the rest of the components (namely, $w$) of vector $\overline{w}$ can be found by solving the equation:
\begin{displaymath}
\overline{df}^{T}\overline{w}=
\left(
\begin{BMAT}(@,2pt,2pt){cc}{c}
df^{T} & \rho
\end{BMAT}
\right)
\left(
\begin{BMAT}(@,2pt,2pt){c}{cc}
w\\
-w_{1}
\end{BMAT}\\
\right)=0,
\end{displaymath}
or, in a compact form,
\begin{equation}\label{eqn:dfw}
df^{T}w=\rho w_{1}.
\end{equation}

\end{proof}

\begin{proof}[Proof of Theorem \ref{thm:stressReduced}]
Following the same construction as in the proof outline above, we first add one of the $K$ edges to $G_{A}$. From \ref{eqn:dfw}, we then compute a stress vector for the extended framework, whose last entry, corresponding to the stress on the added edge, equals to the sign-inverted stress on the same edge in $G_{B}$. Repeating for all $K$ edges, we obtain $K$ stress vectors $\overline{w}_{k}$ by appropriately choosing $w_{1}$ from $\Omega_{B}$ and solving \ref{eqn:dfw} for each edge. From each $\overline{w}_{k}$, in turn, we form a stress matrix such that the $(i_{k},j_{k})$ entry, corresponding to the added edge, is equal to the last (non-zero) entry of $\overline{w}_{k}$. By summing all $K$ stress matrices, we get a stress matrix $\overline{\Omega}_{AK}$ for the framework $\overline{G}_{A}$ constructed by adding \emph{all} $K$ edges to $G_{A}$. The elements of $\overline{\Omega}_{AK}$ corresponding to the added edges will have the same stresses, but with the opposite sign, as the stresses on the same edges in $G_{B}$. However, this matrix in general may neither be PSD nor have nullity $d+1$. From other side, the given matrix $\Omega_{A}$, which is a stress matrix for both $G_{A}$ and $\overline{G}_{A}$, is a PSD matrix with nullity $d+1$, with the entries corresponding to the $K$ added edges equal to 0. Moreover, since $\Omega_{A}$ and $\overline{\Omega}_{AK}$ both satisfy properties (3) and (4) of Definition \ref{dfn:defStressMat}, while $\Omega_{A}$ also satisfies Lemma \ref{lma:projker}, we have $\ker(\Omega_{A})\subseteq\ker(\overline{\Omega}_{AK})$. Therefore by Lemma \ref{lma:omegaepsilon}, with a suitable choice of $c>0$ (such as $c=\frac{2\|\overline{\Omega}_{AK} \|}{\lambda_{[v_{A}-d-1]}(\Omega_{A})}$), the matrix $c\Omega_{A}+\overline{\Omega}_{AK}$ is positive semidefinite with nullity $d+1$, and has the same stresses on the added edges as $\overline{\Omega}_{AK}$.

Finally, we attach the frameworks $\overline{G}_{A}$ and $G_{B}$.
The $K$ added edges in $\overline{G}_{A}$ were already in $G_{B}$, and so this attachment is identical to the regular framework attachment of $G_{A}$ and $G_{B}$. In terms of matrices, however, this attachment corresponds to adding extended $\Omega_{B}$ to the extension of $c\Omega_{A}+\overline{\Omega}_{AK}$, resulting, as in the proof of Theorem \ref{thm:OmegaMain}, in a PSD stress matrix of nullity $d+1$:
\begin{displaymath}
\widetilde\Omega_{re}=c\widetilde\Omega_{A}+\widetilde{\Omega}_{AK}+\widetilde\Omega_{B}.
\end{displaymath}
Its entries, corresponding to the $K$ edges, are equal to $0$, therefore $\widetilde\Omega_{re}$ is the desired stress matrix for the attachment with the $K$ edges removed.
\end{proof}
\end{section}


\begin{section}{Generation of new universally rigid frameworks}\label{sec:Apps}

In \cite{aklat}, Alfakih \emph{et al.} proved that $(d+1)$-lateration framework in general position in $\mathbb{R}^d$ is not only universally rigid but also has a PSD stress matrix of nullity $d+1$.
In this section we will show that edge-reduced framework attachment is a generalization of $(d+1)$-lateration, which implies that a corresponding PSD stress matrix can be constructed as in Theorem \ref{thm:stressReduced}. We will also show that, in general, the edge-reduced attachment can be used to generate universally rigid graphs.
\begin{definition}\label{dfn:d1lateration}
A graph $G=(V,E)$ is a $(d+1)$-lateration graph if:
\begin{enumerate}
\item verices $\{1,\ldots ,d+1\}$ form a complete graph
\item each vertex $i>d+1$ connects to $d+1$ vertices from the set $\{1,\ldots ,i-1\}$.
\end{enumerate}
\end{definition}
The definition above suggests that $(d+1)$-lateration graphs, after appropriately labeling the vertices, can be generated inductively. Starting with a complete graph $G$, at each induction step we expand $G$ by adding a new vertex and connecting it to some $d+1$ existing vertices. The following proposition shows that the same can be accomplished using edge-reduced graph attachment.
\begin{proposition}\label{prop:attachforlat}
Any $(d+1)$-lateration graph can be generated by a sequence of edge-reduced attachments.
\end{proposition}
\begin{proof}
Assume at $k^{th}$ step we introduce a new vertex $v_{k}$ in graph $G$, and want to connect it to $d+1$
vertices $v_{i_{1}},\ldots ,v_{i_{d+1}}$ that are already in $G$. This can be done by first attaching graph G to a complete
graph consisting of vertices $v_{i_{1}},\ldots ,v_{i_{d+1}},v_{k}$ such that $v_{i_{1}},\ldots ,v_{i_{d+1}}$ are shared, and then removing
the edges between the shared vertices that were not in $G$.
\end{proof}

To obtain a PSD stress matrix of nullity $d+1$ for a $(d+1)$-lateration framework in general position we can then use Theorem \ref{thm:stressReduced}, after checking the condition that any such framework is infinitesimally rigid. For this we use the following re-phrased result by Connelly (\cite{cnbook}, Proposition 2.21).
\begin{proposition}[Connelly \cite{cnbook}]\label{prop:inf221}
Suppose $G(\mathbf{p})$ is an infinitesimally rigid framework in general position in $\mathbb{R}^d$. Extend the framework by adding a new vertex $v_{k}$ and connecting it to $d$ existing vertices $v_{i_{1}},\ldots ,v_{i_{d}}$ such that $\mathbf{p}_{k}$ is not in the affine span of $\mathbf{p}_{i_{1}},\ldots ,\mathbf{p}_{i_{d}}$. Then the extended framework is infinitesimally rigid.
\end{proposition}
From here it immediately follows:
\begin{corollary}\label{cor:d1latinf}
Any $(d+1)$-lateration framework in general position is infinitesimally rigid.
\end{corollary}
With this result we can apply Theorem \ref{thm:stressReduced} and compute a PSD stress matrix of nullity $d+1$ for a $(d+1)$-lateration framework in general position. Since the framework generation involves attachment of a complete graph (containing a new vertex) and then removal of some $K$ edges between the shared vertices, each step would require computing a PSD stress matrix of nullity $d+1$ for the complete framework, and then solving $K$ systems of linear equations as described in the theorem.

The edge-reduced attachment can also be applied for generation of new universally rigid graphs from smaller and simpler graphs whose frameworks in general position are always universally rigid. An example of such universally rigid graph, which can not be generated by $(d+1)$-lateration, is shown in Figure \ref{fig:example3lat}. It turns out that all frameworks of the graphs generated by the edge-reduced attachment have PSD stress matrices of nullity $d+1$, as the following theorem shows.
\begin{figure}
\includegraphics[scale=0.7]{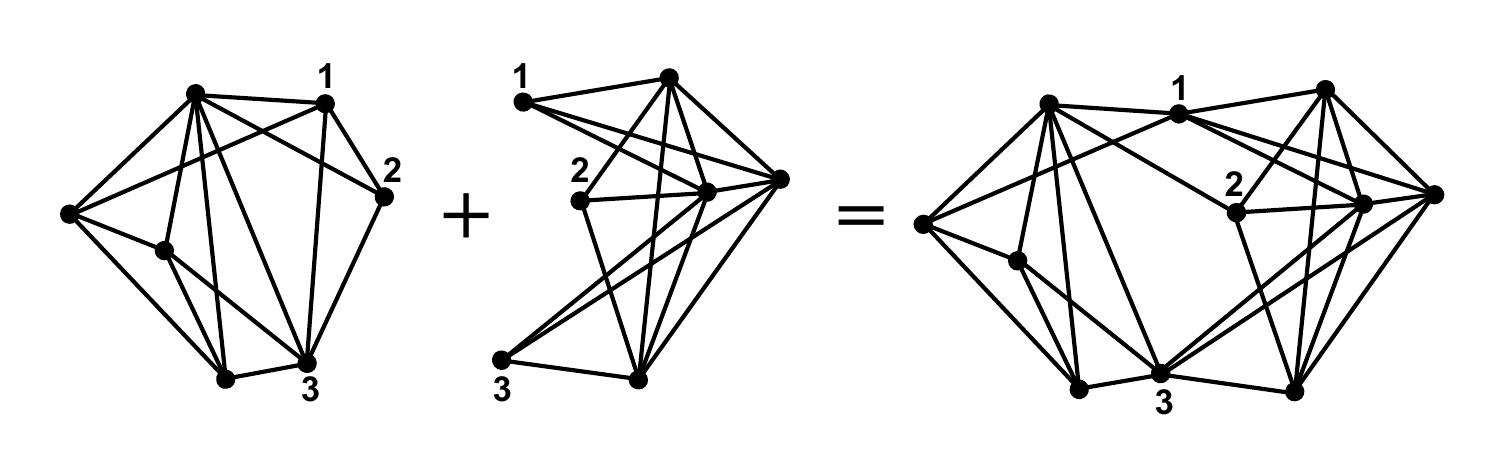}
\caption{Edge-reduced graph attachment of two universally rigid graphs. Each of the two graphs can be generated by $3$-lateration, which is not the case for their edge-reduced attachment.}
\label{fig:example3lat}
\end{figure}
\begin{theorem}\label{thm:URgraph}
Assume we have two frameworks in general position corresponding to two universally rigid graphs. If both frameworks have PSD stress matrices of nullity $d+1$, their edge-reduced attachment also has one.
\end{theorem}
\begin{proof}
To apply Theorem \ref{thm:stressReduced}, we only need to prove that any of the two frameworks is infinitesimally rigid.
Assume that a framework $G(\mathbf{p})$ in general position of such universally rigid graph $G$ is not infinitesimally rigid. Then there
exists a non-trivial infinitesimal flex $\mathbf{q}$. By Corollary \ref{cor:genposopen}, there is $t>0$ such that
the two frameworks $G(\mathbf{p}+t\mathbf{q})$ and $G(\mathbf{p}-t\mathbf{q})$ are in general position. From $(\mathbf{p}_i-\mathbf{p}_j,\mathbf{q}_i-\mathbf{q}_j)=0$ we have
\begin{displaymath}
\|(\mathbf{p}_{i}+t\mathbf{q}_{i})-(\mathbf{p}_{j}+t\mathbf{q}_{j})\|^{2}=
\|\mathbf{p}_i-\mathbf{p}_j\|^{2}+\|t\mathbf{q}_i-t\mathbf{q}_j\|^{2}=
\|(\mathbf{p}_{i}-t\mathbf{q}_{i})-(\mathbf{p}_{j}-t\mathbf{q}_{j})\|^{2},
\end{displaymath}
which holds for any edge $\{i,j\}$. Therefore, since the flex $t\mathbf{q}$ is non-trivial, $G(\mathbf{p}+t\mathbf{q})$ and $G(\mathbf{p}-t\mathbf{q})$
are equivalent but not
congruent, which contradicts universal rigidity of the graph $G$.
\end{proof}
\end{section}


\end{document}